\documentclass[12pt]{amsproc}

\usepackage{tikz}

\usepackage{amsmath}
\usepackage{tikz}
\usepackage{tikzpagenodes}
\usepackage{verbatim}

\usepackage{amssymb}
\usepackage{amsxtra}

\usepackage{amsthm}
\usepackage{latexsym}
\usepackage{ifthen}
\usepackage{mathrsfs}

\def\vt{{\vartheta}}

\def\N{{{\Bbb N}}}
\def\Z{{{\Bbb Z}}}
\def\T{{{\Bbb T}}}
\def\R{{\Bbb R}}
\def\C{{\Bbb C}}
\def\l{{\lambda }}
\def\a{{\alpha }}
\def\D{{\Delta }}

\def\a{{\alpha}}
\def\b{{\beta}}
\def\d{{\delta}}
\def\e{{\varepsilon}}

\def\vp{{\varphi}}
\def\t{{\theta }}
\def\g{{\gamma }}
\def\w{{\omega }}

\def\i{{\mathrm{i} }}

\def\){\right)}
\def\({\left(}

\def\Re{\operatorname{Re}}
\def\Im{\operatorname{Im}}

\def\Re{\operatorname{Re}}
\def\Im{\operatorname{Im}}
\def\var{\operatorname{var}}

\def\const{\operatorname{const}}

\numberwithin{equation}{section}
\newtheorem{lemma}{Lemma}[section]
\newtheorem{theorem}{Theorem}[section]

\newtheorem{remark}{Remark}[section]

\par

\begin{document}

\title[On moduli of smoothness and averaged differences]{On moduli of smoothness and \\ [3pt] averaged differences of fractional order}

\author[Yurii
Kolomoitsev]{Yurii
Kolomoitsev$^{1}$}

\thanks{$^1$Universit\"at zu L\"ubeck,
Institut f\"ur Mathematik,
Ratzeburger Allee 160,
23562 L\"ubeck}



\thanks{E-mail address: kolomoitsev@math.uni-luebeck.de}

\date{\today}

\begin{abstract}
We consider two types of fractional integral moduli of smoothness, which are widely used in theory of functions and approximation theory. In particular, we obtain new equivalences between these moduli of smoothness and the classical moduli of smoothness. It turns out that for fractional integral moduli of smoothness some pathological effects arise.
\end{abstract}

\maketitle

 \section{Introduction}\label{sec:1}

\setcounter{section}{1}
\setcounter{equation}{0}\setcounter{theorem}{0}

Let $\T\cong[0,2\pi)$ be the circle. As usual, the space
$L_p(\T)$, $0<p<\infty$, consists of measurable
functions which are $2\pi$-periodic  and
$$
\Vert f\Vert_p=\bigg(\frac1{2\pi}\int_{\T}|f(x)|^p \mathrm{d}x\bigg)^{\frac 1p}<\infty.
$$
For simplicity, by $L_\infty(\T)$ we denote the space of all
$2\pi$-periodic continuous functions on $\T$ which is equipped
with the norm
$$
\Vert f\Vert_\infty=\max_{x\in\T}|f(x)|.
$$

The classical (fractional) modulus of smoothness of a function $f\in L_p(\T)$, $0< p\le\infty$,
of order $\b>0$ and step $h>0$ is defined by
\begin{equation}\label{eqmod1}
    \w_\b(f,h)_p=\sup_{0<\d<h} \Vert \D_\d^\b f\Vert_p,
\end{equation}
where
$$
\D_\d^\b f(x)=\sum_{\nu=0}^\infty\binom{\b}{\nu}(-1)^\nu f(x+\nu\d),
$$
$
\binom{\b}{\nu}=\frac{\b (\b-1)\dots (\b-\nu+1)}{\nu!},\quad
\binom{\b}{0}=1.
$

For solving particular problems, the usage of the classical modulus of smoothness may be technically very difficult or its application cannot give sharp and meaningful results. Therefore, it arises the necessity to employ modifications of the classical moduli of smoothness. Thus, in the papers~\cite{DHI}, \cite{K11}, \cite{K}, \cite{Pu}, \cite{Rup}, \cite{T2}, \cite{Tumb} different means of averaging of finite differences and their modifications have been studied and applied for solving several problems of approximation theory. As a rule, to construct  such modifications (special moduli of smoothness) one replaces the shift operator $\tau_h f(x)=f(x+h)$ by some smoothing operator, for example, by the Steklov means.

In this paper, we consider two types of special moduli of smoothness given by the following formulas:
\begin{equation}\label{eqmod2}
    {\rm w}_\b(f,h)_p=\(\frac 1h\int_0^h \Vert \D_\d^\b f\Vert_p^{p_1} {\rm d}\d\)^\frac1{p_1}
\end{equation}
and
\begin{equation}\label{eqmod3}
\widetilde{\w}_\b(f,h)_p=\bigg\Vert\frac 1h \int_0^h \D_\d^\b
f(\cdot) {\rm d}\d \bigg\Vert_p,
\end{equation}
where $\b>0$, $h>0$, and $p_1=\min(1,p)$.

Sometimes, the moduli~\eqref{eqmod2} and~\eqref{eqmod3}  are called the integral moduli of smoothness or averaged differences.
The modulus ${\rm w}_\b(f,h)_p$ is well known and it has been often applied for solving different problems of approximation theory, see  e.g.~\cite[Ch. 6, \S 5]{DL}, \cite{DHI}, \cite{Pu}, \cite{Rup}. The modulus $\widetilde{\w}_\b(f,h)_p$ has been introduced and studied in~\cite{Tvinity}, see also~\cite{Tumb} and~\cite[Ch.8]{TB}, in which it is also called the linearized modulus of smoothness. Some applications of the modulus $\widetilde{\w}_\b(f,h)_p$ as well as some of its modifications can be found in~\cite{Ak}, \cite{KS}, \cite{K}, \cite{KT}, and~\cite{T2}.

Note that \eqref{eqmod2} has sense for all function $f\in L_p(\T)$ with $0<p\le \infty$, while \eqref{eqmod3} can be defined only for integrable functions $f$. At the same time,  \eqref{eqmod3} has some advantages. One of them concerns the direct application of the method of Fourier multipliers.

Recall that a numerical sequence
$\{\l_k\}_{k\in\Z}$ is a Fourier multiplier in $L_p(\T)$, if for all functions $f\in L_p(\T)$, $1\le p\le\infty$, the series
$$
\sum_{k\in\Z} \l_k \widehat{f}_k e^{\i kx},\quad \widehat{f}_k=\frac1{2\pi}\int_0^{2\pi} f(x)e^{-\i kx}{\rm d}x,
$$
is the Fourier series of a function $\Lambda f\in L_p(\T)$ and
$$
\Vert\{\l_k\} \Vert_{M_p}=\Vert \Lambda\Vert_{L_p\mapsto L_p}=\sup_{\Vert f\Vert_p\le 1}\Vert \Lambda f\Vert_p <
\infty
$$
(see, e.g.,~\cite[Ch. I and Ch. VI]{SW}).

Let us illustrate how the method of Fourier multipliers can be used in relation to the modulus $\widetilde{\w}_\b(f,h)_p$. It is easy to see that the Fourier series of the averaged difference $\frac 1h \int_0^h \D_\d^\b
f(x){\rm d}\d$ can be written as follows
\begin{equation}\label{mainf0}
\sum_{k\in\Z} \psi_\b(k h)\widehat{f}_k e^{\i kx},
\end{equation}
where the function $\psi_\b$ is defined by
\begin{equation*}
\psi_\b(t)=\int_0^1 (1-e^{\i t\vp})^\b {\rm d}\vp
\end{equation*}
(here and throughout, we use the principal branch of the logarithm). Thus, if we want to obtain, for example, an inequality of the form
$$
\bigg\Vert \sum_{k\in\Z} \l_k(h) \widehat{f}_k e^{\i kx} \bigg\Vert_p\le C(p,\b)\widetilde{\w}_\b(f,h)_p,
$$
we need only to verify that the sequence $\{\l_k(h)/\psi_\b(hk)\}_{k\in\Z}$ is a Fourier multiplier in $L_p(\T)$ and
$$
\sup_{h>0} \bigg\Vert \bigg\{\frac{\l_k(h)}{\psi_\b(kh)}\bigg\}\bigg\Vert_{M_p}\le C(p,\b),
$$
see also Lemma~\ref{lem3} below. This method  has been used, e.g.,  in~\cite{K} and \cite{T2}, see also \cite[Ch.~8]{TB}.

In approximation theory, it is important to ascertain whether a special modulus of smoothness is equivalent to the classical modulus of smoothness $\w_\b(f,h)_p$. For the moduli of smoothness \eqref{eqmod2} and \eqref{eqmod3} with $\b\in \N$, this problem is well studied. In particular, for all $f\in L_p(\T)$, $1\le p\le \infty$, $\b\in \N$, and $h>0$, we have
\begin{equation}\label{equivmmm}
{\rm w}_\b(f,h)_p\asymp {\w}_\b(f,h)_p\asymp\widetilde{\w}_\b(f,h)_p,
\end{equation}
where $\asymp$ is a two-sided inequality with positive constants independent of  $f$ and $h$.
The equivalence ${\rm w}_\b(f,h)_p\asymp {\w}_\b(f,h)_p$, which also holds in the case $0<p<1$, was proved in~\cite[p. 185]{DL}, see also~\cite[Theorem 1]{Pu} and~\cite[Theorem 3.1]{Rup}. The proof of the second equivalence in~\eqref{equivmmm} can be found in~\cite{Tumb} (see also~\cite[Theorem 8.4.1]{TB} for similar results in the Hardy spaces~$H_p$).

The main purpose of this paper is to investigate relations~\eqref{equivmmm} for positive $\b\not \in \N$.
Is easy to see that for any $\b>0$ and $1\le p\le \infty$, by Minkovsky's inequality and trivial estimates, we have
\begin{equation}\label{lemmm}
\widetilde{\w}_\b(f,h)_p\le {\rm w}_\b(f,h)_p\le {\w}_\b(f,h)_p.
\end{equation}
Concerning the inverse inequalities, we have an unexpected result. In Theorem~\ref{th3} below, we show that there exist $f_0(x)\not\equiv\const$, $\b_0>0$, and $h_0>0$ such that
\begin{equation}\label{0}
\widetilde{\w}_{\b_0}(f_0,h_0)_p=0.
\end{equation}
Since $\w_\b(f,h)_p>0$ for all $f(x)\not\equiv\const$ and $h>0$,~\eqref{0} implies that for any $C>0$ and particular $\b>0$
the inequality
\begin{equation*}
{\w}_\b(f,h)_p\le C\widetilde{\w}_\b(f,h)_p
\end{equation*}
does not hold in general.
At the same time, we have a standard situation for ${\rm w}_\b(f,h)_p$: for all $f\in L_p(\T)$, $0<p\le \infty$, $\b>0$, and $h>0$
$$
{\w}_\b(f,h)_p\le C(p,\b){\rm w}_\b(f,h)_p
$$
(see Theorem~\ref{th1} below).

In the paper, we propose several ways to overcome the pathological property~\eqref{0}. In particular, we show that the following modification of~\eqref{eqmod3} can be used instead of the modulus $\widetilde{\w}_\b(f,h)_p$,
$$
\w_\b^*(f,h)_p=\bigg\Vert\frac 1{h^2} \int_0^h
 \int_0^h \D_{\d_1}^{\lfloor\b\rfloor} \D_{\d_2}^{\{\b\}} f(\cdot) {\rm d}\d_1\,
 {\rm d}\d_2\bigg\Vert_p,
$$
where $\lfloor\b\rfloor$ and $\{\b\}$ are the floor and the fractional part functions of $\b$, respectively. In Theorem~\ref{th5} below, we prove that  $\w_\b^*(f,h)_p$ is equivalent to the classical modulus of smoothness for all $\b>0$. At the same time, $\w_\b^*(f,h)_p$ is a convenient modulus in the sense of  applications of Fourier multipliers.

Finally, we note that property~\eqref{0} seems to be very unnatural for moduli of smoothness. However, even in the study of the approximation of functions by some classical methods, for example by Bernstein-Stechkin polynomials, it has been arisen the necessity to construct special moduli of smoothness for which a condition of type \eqref{0} holds (see~\cite{T2}). One has a similar situation  for some non-classical methods of approximation (see~\cite{KT}).

The paper is organized as follows. In Section 2, we present the main results. In Section 3, we formulate and prove auxiliary results. In Section~4, we prove the main results of the paper.

We denote by $C$ some positive constant depending on the indicated parameters. As usual, $A(f, h)\asymp B(f, h)$ means that there exists a positive constant $C$ such that $C^{-1} A(f,h)\le B(f,h)\le C A(f,h)$ for all $f$ and $h$.


\section{Main results}\label{sec:2}

\setcounter{section}{2}
\setcounter{equation}{0}\setcounter{theorem}{0}

We start from the modulus ${\rm w}_\b(f,h)_p$, for which we have a quite standard result in the following theorem.

\begin{theorem}\label{th1}
    Let $f\in L_p(\T)$, $0<p\le\infty$, $\b\in \N\cup (1/p_1-1,\infty)$, and $h\in (0,1)$. Then
\begin{equation*}
{\rm w}_\b(f,h)_p\asymp {\w}_\b(f,h)_p\,.
\end{equation*}
\end{theorem}

The next theorem is the key result of the paper.

\begin{theorem}\label{th3} There exists a set $\{\b_k\}_{k=0}^\infty$ such that $\b_0\in (4,5)$, $\b_k\to\infty$ as $k\to\infty$, and the following assertions hold:

\smallskip

$(i)$ for each function $f\in L_p(\T)$, $1\le p\le\infty$, and for all $\b\in (0,\b_0)\cup \N$ and $h\in (0,1)$ we have
    $$
     \widetilde{\w}_\b(f,h)_p\asymp \w_\b(f,h)_p\,;
    $$

$(ii)$ for each $k\in \Z_+$ there exists $t_k>2\pi$ such that for any $n\in \Z\setminus \{0\}$ we have
 \begin{equation*}
   \frac 1h\int_0^{h} \D_\d^{\b_k} e_n(x){\rm d}\d=0,\quad x\in \T,
 \end{equation*}
where $e_n(x)=e^{inx}$ and $h=t_k/|n|$. In particular, for all $0<p\le \infty$, we have
\begin{equation}\label{02}
  \widetilde{\w}_{\b_k}(e_n,h)_p=0\,.
\end{equation}
\end{theorem}

Theorem~\ref{th3} implies that, for small values of $\b$, the modulus~\eqref{eqmod3} has the same properties as the classical modulus of smoothness~\eqref{eqmod1}, but for some $\b >4$  the modulus~\eqref{eqmod3} has a pathological behaviour.

In the next results, we show several ways to overcome the effect of~\eqref{02}.
In particular problems of approximation theory, it is enough to know the behaviour of moduli of smoothness on a set of trigonometric polynomials. It turns out that one can reduce the influence of property~\eqref{02} in such situation.

Let $\mathcal{T}_n$ be the set of all trigonometric
polynomials of order at most $n$,
\begin{equation*}
\mathcal{T}_n=\bigg\{T(x)=\sum_{\nu=-n}^n c_\nu {\rm e}^{\i\nu x}~:~
c_\nu \in \C\bigg\}.
\end{equation*}

\begin{theorem}\label{th4} Let $1\le p\le\infty$, $\b>0$, and $n\in\N$. Then for each $T_n\in \mathcal{T}_n$ and for each $h\in
(0,1/n)$ we have
    $$
     \widetilde{\w}_\b(T_n,h)_p\asymp \w_\b(T_n,h)_p.
    $$
\end{theorem}

\begin{remark}
Theorem~\ref{th4} is also true in the case $0<p<1$. This follows from the proof of Theorem~\ref{th4} presented below and the corresponding results in~\cite{K07}.
\end{remark}

Another way to reduce the effect of~\eqref{02} is adding to  $\widetilde{\w}_\b(f,h)_p$ a quantity that has a better behaviour than the classical modulus of smoothness. For this purpose, one may use the error of the best approximation. As usual, the error of the best approximation of a function $f$ in $L_p$ by trigonometric polynomials of order at most $n$ is given by
$$
E_n(f)_p=\inf_{T\in \mathcal{T}_n}\Vert f-T\Vert_p.
$$

Recall the well-known Jackson inequality:
{\it for all $f\in L_p$, $0< p\le\infty$, $r\in \N$, and $n\in\N$ we have
\begin{equation}\label{eqI12}
  E_n(f)_p\le C\w_r\(f,\frac 1n\)_p,
\end{equation}
where $C$ is a constant independent of $f$ and $n$}
(see~\cite{St51} for $1\le p\le\infty$ and~\cite{SO} for $0<p<1$; see also~\cite{I},~\cite{Rup}, and \cite{SKO}).

\begin{theorem}\label{thbest}
      Let $f\in L_p(\T)$, $1\le p\le\infty$, $\b>0$, and $h\in (0,1)$. Then
\begin{equation*}
{\w}_\b(f,h)_p\asymp \widetilde{\w}_\b(f,h)_p+E_{[1/h]}(f)_p\,.
\end{equation*}
\end{theorem}

Finally, let us consider a modification of~\eqref{eqmod3}.
Let $\b>\a>0$. Denote
\begin{equation*}
    \w_{\b;\, \a}^*(f,h)_p=\bigg\Vert\frac 1{h^2} \int_0^h
 \int_0^h \D_{\d_1}^{\b-\a} \D_{\d_2}^{\a} f(\cdot) {\rm d}\d_1
 {\rm d}\d_2\bigg\Vert_p.
\end{equation*}
It is easy to see that $\b$ is the main parameter in the above modulus of smoothness.

\begin{theorem}\label{th5} Let $f\in L_p(\T)$, $1\le p\le\infty$, and $\b>0$ and $\a \in (0,4]$ be such that $\b-\a\in \Z_+$. Then for all $h\in (0,1)$ we have
\begin{equation*}
  \w_{\b;\, \a}^*(f,h)_p\asymp \w_\b(f,h)_p\,.
\end{equation*}
\end{theorem}

\section{Auxiliary results}

\setcounter{section}{3}
\setcounter{equation}{0}\setcounter{theorem}{0}

\subsection{Properties of the differences and moduli of smoothness.} Let us recall several basic properties of the differences and moduli of smoothness of fractional order (see~\cite{ But, DL, samko,  tabeR}).

For  $f,f_1,f_2\in L_p(\T)$, $0< p\le \infty$, $h>0$, and $\a,\b \in\N\cup \(1/{p_1}-1,\infty\)$, we have
\begin{itemize}
\item[{$(a)$}]
 $\D_h^\b(f_1+f_2)(x) = \D_h^{\b} f_1(x) + \D_h^{\b} f_2(x)$;

\item[{$(b)$}]
 $\D_h^\alpha(\D_h^\beta f)(x) = \D_h^{\alpha+\beta} f(x)$;

\item[{$(c)$}]  $\|\D_h^{\b} f\|_{p} \le C(\b,p)\|f\|_{p}$;
\end{itemize}

\begin{itemize}
\item[${(d)}$]
$ \omega_\beta(f,\delta)_p$ is a non-negative non-decreasing function of $\delta$ such that
$\lim_{\delta\to 0+} \omega_\beta(f,\delta)_p=0;$
\item[{$(e)$}]
$              \omega_\beta(f_1+f_2,\delta)_p\le 2^{\frac1{p_1}-1}
\big(     \omega_\beta(f_1,\delta)_p+\omega_\beta(f_2,\delta)_p\big);
$
 \item[{$(f)$}]
$              \omega_{\alpha+\beta}(f,\delta)_p\le
C(\alpha,p) \omega_{\beta}(f,\delta)_p$;
\item[{$(g)$}]
for $\lambda\ge 1$ and $\b\in\N$,
        $$
        \omega_{\beta}(f,\lambda \delta)_p\le C(\beta,p) \lambda^{\beta+\frac1{p_1}-1}  \omega_{\beta}(f,\delta)_p.
        $$
\end{itemize}

We will  use the following Boas type inequality.

\begin{lemma}\label{lemNS}
Let $0<p\le\infty$, $\beta>0$, $n\in\mathbb{N}$, and
$0<\delta, h\le {\pi}/{n}$. Then for each
$T_n\in\mathcal{T}_n$ we have
\begin{equation*}
h^{-\beta}\Vert\Delta_h^\beta T_n\Vert_p\asymp\delta^{-\beta}\Vert\Delta_\delta^\beta T_n\Vert_p,
\end{equation*}
where $\asymp$ is a two-sided inequality with absolute constants independent of $T_n$, $h$, and $\d$.
\end{lemma}

In the case $1\le p \le\infty$, Lemma~\ref{lemNS} follows from~\cite[4.8.6 and 4.12.18]{timan} (for integer $\beta$), and~\cite{tabeR} (for any positive $\beta$).
In the case $0<p<1$, it follows from~\cite{DHI} (for integer $\beta$) and from~\cite{K07} (for any positive $\beta$).

\subsection{Properties of the function $\psi_\b(t)$.}

Everywhere below we set
$$
z_\b(t)=t\psi_\b(t)=\int_0^t (1-e^{\i\vp})^\b {\rm d}\vp.
$$

The next lemma is a key result for proving Theorem~\ref{th3} $(i)$.

\begin{lemma}\label{lem0} {\sc (See~\cite{Tumb}).}
Let $\b>0$ and $\psi_\b(t)\neq 0$ for $t\in\R\setminus \{0\}$.
Then for each function $f\in L_p(\T)$, $1\le p\le\infty$, and $h>0$ we have
$$
{\w}_\b(f,h)_p\le  C(\b)\widetilde{{\w}}_\b(f,h)_p.
$$
\end{lemma}

The following result was proved in~\cite[8.3.5 b)]{TB} by using Lindemann's classical theorem about the transcendence
of values of the exponential function.

\begin{lemma}\label{lemtran}
  If $\b\in\N$, then $\psi_\b(t)\neq 0$ for all $t\in \R\setminus \{0\}$.
\end{lemma}

Thus, combining Lemma~\ref{lemtran} and Lemma~\ref{lem0}, we get the proof of Theorem~\ref{th3} $(i)$ for $\b\in \N$.
Below, we obtain some unexpected properties of $\psi_\b$ for non-integer~$\b$.

The following lemma is the main auxiliary result for proving Theorem~\ref{th3} in the case $\b\not\in \N$.

\begin{lemma}\label{lem1}
There exists a set $\{\b_k\}_{k=0}^\infty$ such that $\b_0\in
(4,5)$, $\b_k\to\infty$ as $k\to\infty$, and
\begin{itemize}
  \item[$(i)$] for all $\b\in (0,\b_0)$ we have
\begin{equation}\label{eq0.lem1}
    z_\b(t)\neq 0\quad\text{for}\quad t\in \R\setminus \{0\};
\end{equation}
  \item[$(ii)$] for each $k\in \Z_+$ there exists $t_k>2\pi$ such that
  $z_{\b_k}(t_k)=0$;
   \item[$(iii)$] for all $\b\in (0,\infty)$ we have
\begin{equation*}
    z_\b(t)\neq 0\quad\text{for}\quad 0<|t|<\pi.
\end{equation*}
\end{itemize}
\end{lemma}

\begin{proof}

\emph{{The proof of $(i)$.}}
First let us derive basic properties of the function $z_\b$.
By simple calculation, we get
\begin{equation}\label{eq.zv}
    x_\b(t)=\Re z_\b(t)=t+\sum_{\nu=1}^\infty\binom{\b}{\nu}
(-1)^\nu\frac{\sin \nu t}\nu
\end{equation}
and
$$
y_\b(t)=\Im z_\b(t)=\sum_{\nu=1}^\infty\binom{\b}{\nu}
(-1)^\nu\frac{1-\cos \nu t}\nu.
$$
These equalities imply that
\begin{equation}\label{+++}
  x_\b(-t)=-x_\b(t)\quad\text{and}\quad y_\b(-t)=y_\b(t),\quad t\in\R.
\end{equation}
Thus, to prove the lemma, it is enough to consider only the case $t>0$.

It is easy to see that for $t\in [0,2\pi]$ we have
\begin{equation}\label{eq1.lem1}
    x_\b(t)=2\pi-x_\b(2\pi-t)=x_\b(2\pi+t)-2\pi,
\end{equation}
\begin{equation}\label{eq2.lem1}
    y_\b(t)=y_\b(2\pi-t)=y_\b(2\pi+t),
\end{equation}
and
\begin{equation*}
    x_\b(\pi k)=\pi k,\quad y_\b(2\pi k)=0,\quad k\in\Z_+.
\end{equation*}

Denote
$$
\g_{\b,k}=\{z_\b(t),\, 2\pi k\le t< 2\pi (k+1)\}\quad\text{and}\quad
\g_{\b}=\bigcup_{k=0}^\infty \g_{\b,k}.
$$
Equalities (\ref{eq1.lem1}) and (\ref{eq2.lem1}) imply that the curve $\g_{\b,0}$ is symmetric with respect to the line $x=\pi$ on the complex plane $\C$. We also have that
\begin{equation}\label{eq4.lem1}
    \g_{\b,k+1}=\g_{\b,k}+2\pi,\quad k\in\Z_+.
\end{equation}
See on Figure~1  the form of the curves $\g_{\b,k}$ in the cases $\b=4$ and $k=0,1,2$.

Note that for $t\in [0,2\pi]$ the following equalities hold
\begin{equation}\label{eq5.lem1}
    x_\b(t)=\frac 2\b\int_{-\b\pi/2}^{\b(t-\pi)/2} \cos\vp \left(2\cos \frac
    \vp\b\right)^\b {\rm d}\vp
\end{equation}
and
\begin{equation}\label{eq6.lem1}
    y_\b(t)=\frac 2\b\int_{-\b\pi/2}^{\b(t-\pi)/2} \sin\vp \left(2\cos \frac
    \vp\b\right)^\b {\rm d}\vp.
\end{equation}
Therefore,
$$
x_\b'(t)=\cos\frac{\b(t-\pi)}{2}\left(2\sin\frac
    t2\right)^\b
\quad \text{and} \quad
y_\b'(t)=\sin\frac{\b(t-\pi)}{2}\left(2\sin\frac
    t2\right)^\b.
$$
Below, these two equalities  will be often used to indicate intervals of monotonicity  of the functions $x_\b$ and $y_\b$.

By Lemma~\ref{lemtran}, we have that if $\b\in\N$, then $z_\b(t)\neq 0$ for all $t\neq 0$. Now, we show that $z_\b(t)\neq 0$ for $\b\in
(0,4)\setminus \N$ and $t>0$.
We split the proof of this fact into several cases.

The simplest case is $\b \in(0,1)$. Indeed,  by (\ref{eq5.lem1})
and (\ref{eq4.lem1}), we have that $x_\b(t)>0$ for $t>0$.
This obviously implies~\eqref{eq0.lem1}.

The next case $\b\in (1,2)$ is also simple. In this case, by
(\ref{eq6.lem1}), (\ref{eq2.lem1}), and (\ref{eq4.lem1}), we get that $y_{\b}(t)<0$
for all $t\in\R_+\setminus \{2\pi k\}_{k\in\Z_+}$, and $x_{\b}(2\pi
k)=2\pi k$, $k\in\Z_+$. Hence, $z_\b(t)\neq 0$ for all $t>0$.

In what follows we deal only with the case $\b\in (3,4)$. The proof of the lemma in the case $\b\in
(2,3)$ is similar to the arguments presented below.

First, let us consider the curve $\g_{\b,0}$ for $\b\in (3,4)$.
Let
$$
I_1=(0,\pi(1-3/\b)]
$$
and
$$
I_j=(\pi(1-(5-j)/\b),\pi(1-(4-j)/\b)],\quad j=2,3,4.
$$
We are going to show that for each $j\in\{1,2,3,4\}$
\begin{equation}\label{eq7.lem1}
z_\b(t)\neq 0,\,\,\, \text{ $t\in I_j$}.
\end{equation}

\emph{1) {The case $j=1$.}} In this case, (\ref{eq7.lem1}) easily follows from the fact that the functions $x_\b$ and $y_\b$ are strictly increasing and positive on $I_1$.

\emph{2) {The case $j=2$.}}
We have
\begin{equation*}
\begin{split}
   x_\b\(\pi\(1-\frac2\b\)\)&= \frac 2\b\int_{-\b\pi/2}^{-\pi} \cos\vp
\left(2\cos \frac
    \vp\b\right)^\b
    {\rm d}\vp \\
    &<\frac 2\b \left(2\cos \frac{3\pi}{2\b}\right)^\b
\sin \frac{\pi\b}2<0.
\end{split}
\end{equation*}
Thus, the function $x_\b$ is strictly decreasing and changes the sign from "$+$"\, to "$-$"\, on $I_2$. At the same time, the function $y_\b$ is strictly increasing and positive on $I_2$. The last fact implies (\ref{eq7.lem1}) for $j=2$.

\emph{3) {The case $j=3$.}}
Let us show that
$
y_\b\(\pi\(1-\frac1\b\)\)<0.
$
We have
\begin{equation*}
    \begin{split}
\frac\b 2 y_\b&\(\pi\(1-\frac1\b\)\)=\int_{-\b\pi/2}^{-\pi/2} \sin\vp
\left(2\cos \frac
    \vp\b\right)^\b
    {\rm d}\vp\\
    &=\left\{\int_{-\b\pi/2}^{-3\pi/2}+\int_{-3\pi/2}^{-\pi}+\int_{-\pi}^{-\pi/2}\right\}\sin\vp
\left(2\cos \frac
    \vp\b\right)^\b
    {\rm d}\vp\\
    &=S_1+S_2+S_3.
     \end{split}
\end{equation*}
One can estimate the integrals $S_i$, $i=1,2,3$,  by the following way:
$$
S_1<\left(2\cos \frac
    {3\pi}{2\b}\right)^\b \cos \frac
    {\pi\b}{2}<\left(2\cos \frac
    {3\pi}{8}\right)^3=(2-\sqrt{2})^{3/2},
$$
\begin{equation*}
  \begin{split}
     S_2&=\left\{\int_{-3\pi/2}^{-5\pi/4}+\int_{-5\pi/4}^{-\pi}\right\}\sin\vp
\left(2\cos \frac
    \vp\b\right)^\b
    {\rm d}\vp\\
    &<\left(2\cos \frac
    {5\pi}{16}\right)^4\frac1{\sqrt 2}+\left(2\cos \frac
    {\pi}{\b}\right)^\b(1-1/{\sqrt 2}),
   \end{split}
\end{equation*}
and
\begin{equation*}
  \begin{split}
     S_3&=\left\{\int_{-\pi}^{-3\pi/4}+\int_{-3\pi/4}^{-\pi/2}\right\}\sin\vp
\left(2\cos \frac\vp\b\right)^\b{\rm d}\vp\\
&<\left(2\cos \frac
    {\pi}{\b}\right)^\b(1/{\sqrt 2}-1)-\left(2\cos \frac
    {\pi}{4}\right)^3\frac1{\sqrt 2}.
   \end{split}
\end{equation*}
Combining the above estimates for $S_i$, we get
$$
y_\b\(\pi\(1-\frac1\b\)\)<\frac 23\bigg[(2-\sqrt{2})^{3/2}+\frac
1{\sqrt{2}}\bigg(\left(2\cos \frac
    {5\pi}{16}\right)^4-2^{3/2}\bigg)\bigg]<0.
$$

We have that the functions $x_\b$ and $y_\b$ are strictly decreasing,
the function $y_\b$ changes the sign from "$+$"\, to "$-$"\,, but $x_\b$ is negative on $I_3$.
Combining these facts, we get that (\ref{eq7.lem1}) holds for $j=3$.

\emph{4) {The case $j=4$.}}
We have that the function $y_\b$ is strictly decreasing and negative on $I_4$, the function $x_\b$ is strictly increasing and changes the sign from  "$-$"\, to "$+$"\, on this interval. Therefore, $z_\b(t)\neq 0$ for all
$t\in I_4$.

Thus, we have shown that $z_\b(t)\neq 0$ for all $t\in
(0,\pi]$. Taking into account that the curve $\g_{\b,0}$ is symmetric with respect to the line  $x=\pi$, we get that $z_\b(t)\neq
0$ for all $t\in (\pi,2\pi]$, too.

\medskip

Now let us consider the curves $\g_{\b,k}$ for $k\in\N$. In view of~\eqref{eq4.lem1}, to finish the proof of part $(i)$
we need to show that
\begin{equation}\label{eq12.lem1}
    \g_{\b,1}\in \{z\in\C\,:\, \Re z>0\}.
\end{equation}
Note that $\g_{\b,1}=\g_{\b,0}+2\pi$. Thus, to verify
(\ref{eq12.lem1}) we only need to investigate the  extremal points of the function
$x_\b(t)$ for all $\b\in (3,4)$ and $2\pi\le t\le 4\pi$. From the results obtained above, it follows that the points
$\tau_0=\pi(3-3/\b)$, $\tau_1=\pi(3-1/\b)$, $\tau_2=\pi(3+1/\b)$, and
$\tau_3=\pi(3+3/\b)$ are extremal for $x_\b(t)$. It is easy to see that it is enough to consider
$x_\b$ at the points $\tau_1$ and $\tau_3$ (see Figure~1).

  \begin{center}
  \includegraphics[width=9cm]{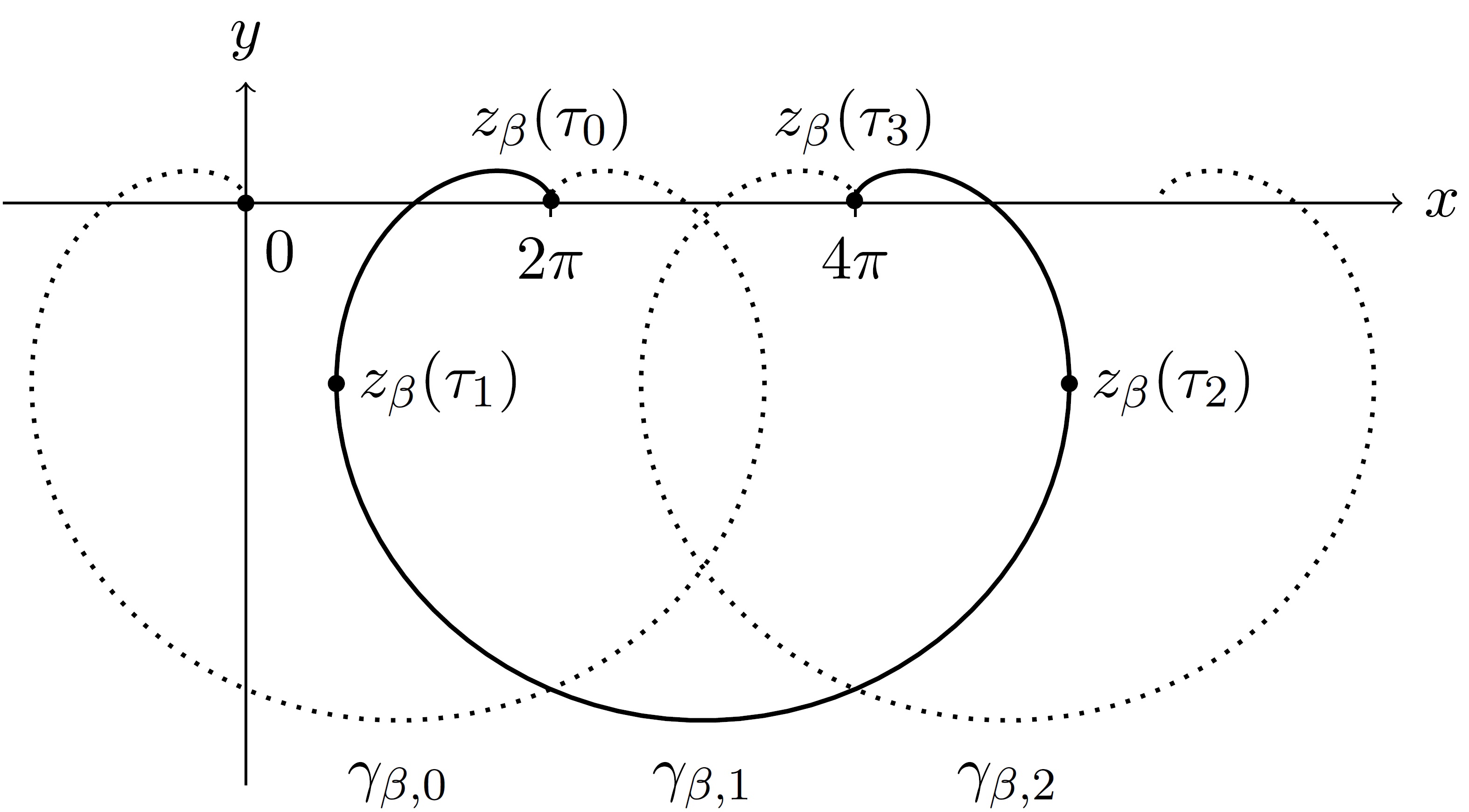}
 \end{center}
 \begin{center}
{\textbf{Figure 1.}}
\end{center}

Simple estimates show that
$x_\b(\tau_3)=4\pi-x_\b(\pi(1-3/\b))>0$. We also have
\begin{equation}\label{eq13.lem1}
x_\b(\tau_1)=2\pi+x_\b\(\pi\(1-\frac1\b\)\).
\end{equation}
Using~\eqref{eq5.lem1}, we get
\begin{equation}\label{dopp}
  \begin{split}
     x_\b\(\pi\(1-\frac1\b\)\)&=\left\{\int_{-\pi\b/2}^{-3\pi/2}+\int_{-3\pi/2}^{-\pi/2}\right\}\frac 2\b\cos\vp \left(2\cos \frac
    \vp\b\right)^\b {\rm d}\vp\\
    &=S_1+S_2.
   \end{split}
\end{equation}
It is evident that $S_1>0$. The integral $S_2$ can be estimated by the following way
\begin{equation}\label{eq14.lem1}
\begin{split}
    S_2&=\frac 2\b \sum_{k=2}^5 \int_{-\pi(k+1)/4}^{-\pi k/4}\cos
    \vp\left( 2\cos\frac\vp\b\right)^\b {\rm d}\vp\\
    &>\frac 2\b \sum_{k=2}^5 \left( 2\cos\frac{\pi k}{16}\right)^\b\int_{-\pi(k+1)/4}^{-\pi k/4}\cos
    \vp {\rm d}\vp\\
    &=-\frac 2\b \bigg[\bigg(1-\frac1{\sqrt 2}\bigg)\bigg\{\left( 2\cos\frac{5\pi}{16}\right)^\b+\left( 2\cos\frac{\pi }{8}\right)^\b\bigg\}\\
    &\quad\quad\quad\quad\quad\quad\quad\quad\quad+
    \frac1{\sqrt 2}\bigg\{\left( 2\cos\frac{\pi}{4}\right)^\b+\left( 2\cos\frac{3\pi
    }{16}\right)^\b\bigg\}\bigg].
\end{split}
\end{equation}
From~\eqref{eq14.lem1}, taking into account that the function $f(x)={a^x}/{x}$ is increasing for $a>1$ and $x>1$, we get
\begin{equation*}
  \begin{split}
    S_2>-\frac 12 \bigg[\bigg(1-\frac1{\sqrt 2}\bigg)\bigg\{&\left(
2\cos\frac{5\pi}{16}\right)^4+(\sqrt{2}+2)^2\bigg\}\\
&+
    \frac1{\sqrt 2}\bigg\{2^2+\left( 2\cos\frac{3\pi
    }{16}\right)^4\bigg\}\bigg]>-2\pi.
   \end{split}
\end{equation*}
The last estimate together with (\ref{dopp}) and (\ref{eq13.lem1}) implies that $x_\b(\tau_1)>0$. Therefore, we have (\ref{eq12.lem1}).

Finally, combining (\ref{eq12.lem1}) and  (\ref{eq4.lem1}), we prove the first part of Lemma~\ref{lem1}.

\medskip

\emph{{The proof of $(ii)$.}}
To prove the second part of the lemma, it is enough to investigate the curve
$$
\Gamma_{\b,1}=\{z_\b(t)\,:\, \pi(3-2/\b)\le t\le
3\pi\}\quad\text{for}\quad \b\in[4,5]
$$
(see Figure 2 below).

By analogy with the proof of the first part of the lemma, taking into account the equality
$$
\Gamma_{\b,1}=\{z_\b(t)\,:\, \pi(1-2/\b)\le t\le \pi\}+2\pi,
$$
we get
\begin{equation}\label{eqDop.lem1}
    y_\b(\pi(3-2/\b))>0\quad\text{and}\quad y_\b(3\pi)<0\quad\text{for all}\quad  \b\in[4,5].
\end{equation}
Moreover, the functions $x_\b$ and $y_\b$ are strictly decreasing on $[\pi(3-2/\b),\\\pi(3-1/\b))$. At the same time, we see that the function $x_\b$ is strictly increasing and $y_\b$ is strictly decreasing on  $[\pi(3-1/\b),3\pi]$.
Combining these facts, we have that the curve  $\Gamma_{\b,1}$ intersects the line $y=0$ only once. Thus, one can define the function  $\vt:[4,5]\mapsto
(\pi(3-2/\b),3\pi)$ by the rule
\begin{equation}\label{*}
  {y}_\b|_{[\pi(3-2/\b),3\pi]}(\vt(\b))=0,
\end{equation}
where ${y}_\b|_A$ denotes a restriction of the function  $y_\b$ on some set
$A$.

Let us consider the function
$$
F(\b)=x_\b(\vt(\b)).
$$
We need to verify that $F(\b)$ is a continuous function on the interval $(4,5)$. First, let us show that the function $\vt(\b)$ is continuous on $(4,5)$.

Let $4\le\b'<\b''\le 5$. Without loss of generality, we can suppose that $\vt(\b')<\vt(\b'')$. Using~\eqref{*}, \eqref{eq2.lem1}, and~\eqref{eq6.lem1}, we obtain
\begin{equation}\label{**}
  \begin{split}
  y_{\b'}(\vt(\b''))-y_{\b''}(\vt(\b''))&=y_{\b'}(\vt(\b''))\\
  &=y_{\b'}(\vt(\b''))-y_{\b'}(\vt(\b'))\\
  &=y_{\b'}(\vt(\b'')-2\pi)-y_{\b'}(\vt(\b')-2\pi)\\
  &=\frac2{\b'}\int_{\frac{\b'}2(\vt(\b')-3\pi)}^{\frac{\b'}2(\vt(\b'')-3\pi)}\sin \vp
      \left(2\cos\frac{\vp}{\b'}\right)^{\b'}{\rm d}\vp.
  \end{split}
\end{equation}
By~\eqref{**} and the mean value theorem, there exists a point $\xi\in({\b'}(\vt(\b')$ $-3\pi)/2, {\b'}(\vt(\b'')-3\pi)/2)$ such that
\begin{equation}\label{oo}
  y_{\b'}(\vt(\b''))-y_{\b''}(\vt(\b''))=\sin \xi
      \left(2\cos\frac{\xi}{\b'}\right)^{\b'}\(\vt(\b'')-\vt(\b')\).
\end{equation}
By~\eqref{eqDop.lem1} and continuity of $y_\b(t)$, there exists a sufficiently small $\d>0$ such that
$
  \pi\(3-2/\b\)+\d<\vt(\b)<3\pi-\d.
$
Therefore, we get
\begin{equation*}\label{o}
-\pi+2\d<-\pi+\frac{\b' \d}{2}<\xi<-\frac{\b' \d}{2}<-2\d.
\end{equation*}
Taking into account these inequalities, we obtain
\begin{equation}\label{ooo}
  \left| \sin \xi
      \left(2\cos\frac{\xi}{\b'}\right)^{\b'} \right|> \sin(2\d)\left(2\cos\(\frac\pi4-\frac\d2\)\right)^{4}=C(\d)>0.
\end{equation}
Thus, by~\eqref{oo} and~\eqref{ooo}, we have
\begin{equation}\label{oooo}
|\vt(\b'')-\vt(\b')|<\frac1{C(\d)}|y_{\b'}(\vt(\b''))-y_{\b''}(\vt(\b''))|.
\end{equation}
Since $y_\b(t)$ is a continuous function of $\b$, from~\eqref{oooo}, we get that the function $\vt(\b)$ is also continuous on $(4,5)$.

Now, taking into account that $x_\b(t)$ is continuous on $(4,5)\times (2\pi,3\pi)$, we get  that the function $F(\b)$ is continuous on $(4,5)$.

Next, from~\eqref{eq12.lem1}, we get that $F(4)>0$. Thus, if we show that
$F(5)<0$, then, by the intermediate value theorem,  we can find  $\b_0\in (4,5)$ such that
$F(\b_0)=0$. This and~\eqref{*} will imply that
$z_{\b_0}(f(\b_0))=0$.

Indeed, it is easy to see that the functions $x_5$ and $y_5$ are strictly decreasing on $(13\pi/5,3\pi)$. Thus, to prove the existence of a point $t_0\in (13\pi/5,3\pi)$ such that
\begin{equation*}
   x_{5}(t_0)<0\quad\text{and}\quad y_{5}(t_0)=0,
\end{equation*}
it is enough to verify that the origin of  $\Gamma_{5,1}$ lies in the first quadrant, the extremal point  of the curve (relating to the extremal point of the function $x_5$) lies in the third quadrant, and there exists an intermediate point that lies in the second quadrant.

  \begin{center}
  \includegraphics[width=7cm]{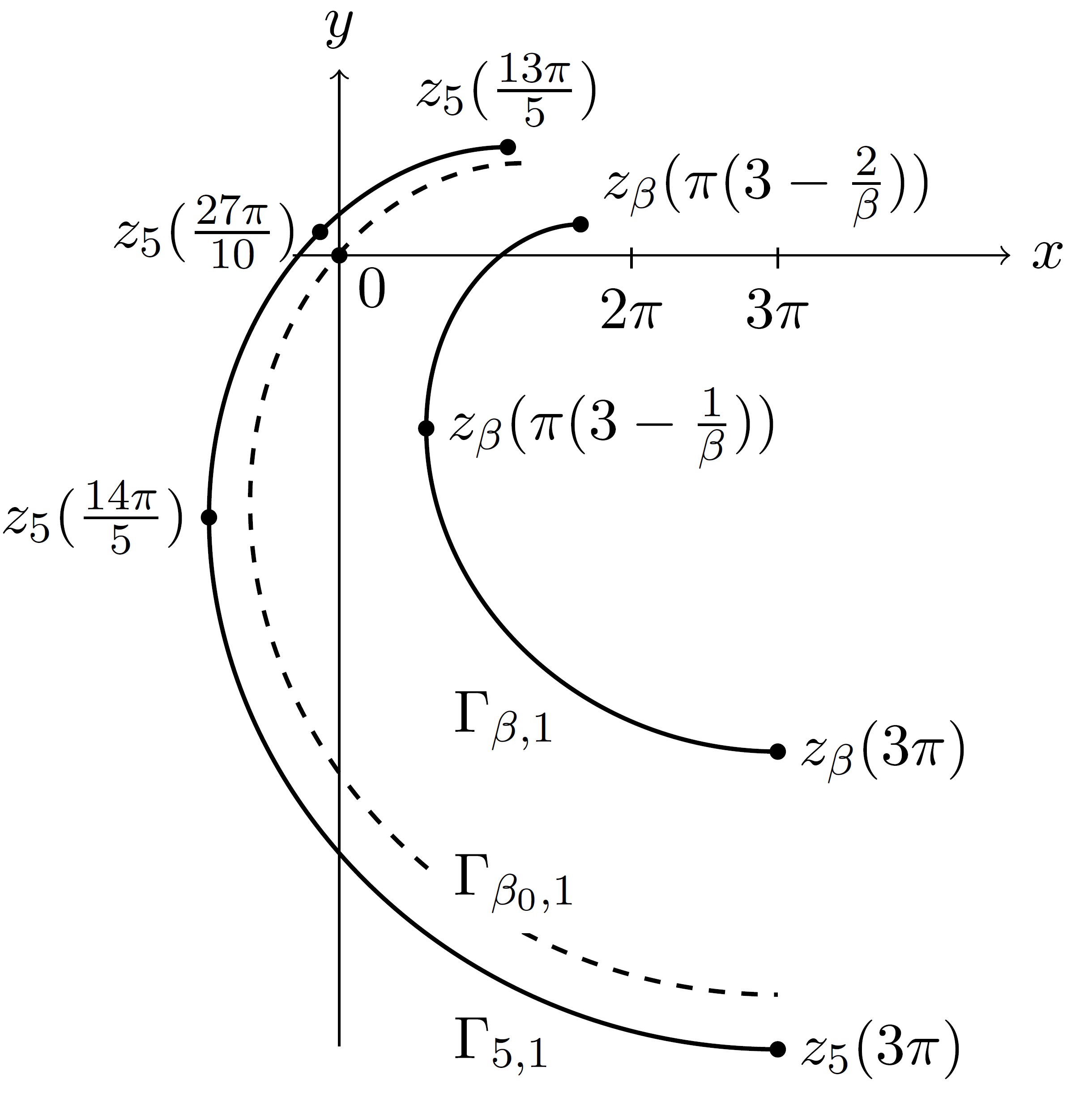}
 \end{center}
\begin{center}
\textbf{Figure 2.}
\end{center}

Performing simple calculation, we can verify that the origin of the curve  $\Gamma_{5,1}$ takes the value
$
z_5\left({13\pi}/5\right)\approx 3.622+{\rm i}2.327,
$
that is it belongs to the first quadrant; since
$
z_5\left({14\pi}/5\right)\approx -2.803-{\rm i}5.632,
$
the point $z_5(14\pi/5)$, which is the extremal point of the curve $\Gamma_{5,1}$, belongs to the third quadrant;  since
$
z_5\left({27\pi}/{10}\right)\approx -0.413+{\rm i}0.504,
$
the intermediate point  $z_5(27\pi/10)$ belongs to the second quadrant (see also Figure 2 in which $4<\b<4.5$ and $\b_0\approx 4.85$).

Now, let us prove assertion $(ii)$ for any $k\in \N$.
First, we show that for each $\b>4$
\begin{equation}\label{verif}
    y_\b\(\pi \(1-\frac 2\b\)\)>0\quad\text{and}\quad y_\b\(\pi\(1-\frac1\b\)\)<0.
\end{equation}
Let us verify the first inequality. The second one can be proved similarly.

Let $k_0$ be the smallest natural number such that  ${\b
\pi}/2\ge (2k_0+1)\pi$. Then, using~\eqref{eq6.lem1}, we have
\begin{equation*}
\begin{split}
      &\frac \b 2 y_\b\(\pi \(1-\frac 2\b\)\)\\
      &=\left\{\int_{-\frac{ \b \pi}{2}}^{-\pi(2k_0+1)}+\sum_{k=k_0}^1 \int_{-(2k+1)\pi}^{-(2k-1)\pi} \right\}
    \sin\vp \left(2\cos \frac
    \vp\b\right)^\b {\rm d}\vp\\
    &=S_1+S_2.
\end{split}
\end{equation*}
It is evident that $S_1\ge 0$. Thus, to verify~\eqref{verif}, it is enough to show that $S_2>0$. But this easily follows  from the fact that for each
$k\in\N$ we have
\begin{equation*}
\begin{split}
      \int_{-(2k+1)\pi}^{-(2k-1)\pi}&\sin\vp \left(2\cos \frac
    \vp\b\right)^\b {\rm d}\vp\\
    &> \(2 \cos\frac{2\pi k}{\b}\)^\b
    \left\{\int_{-(2k+1)\pi}^{-2k \pi}+\int_{-2k\pi}^{-(2k-1) \pi}\right\}\sin\vp {\rm d}\vp=0.
\end{split}
\end{equation*}

By analogy with the proof of part $(i)$, we see that the function $y_\b$ is strictly decreasing and changes the sign from "$+$"\, to "$-$"\, on
$(\pi(1-2/\b),\pi)$ for all $\b>5$. Therefore, as above, for any $\b>5$ one can define a function $\t : [5,\infty)\mapsto (\pi\(1-2/\b\),\pi\(1-1/\b\))$ by the rule
\begin{equation*}\label{*1}
  {y}_\b|_{(\pi\(1-2/\b\),\pi\(1-1/\b\))}(\t(\b))=0.
\end{equation*}

Let us show that for $n\in\N$ one has
\begin{equation}\label{eqkvad}
    x_{2n}\(\pi\(1-\frac 1{n}\)\)\to
    -\infty\quad\text{as}\quad
    n\to\infty.
\end{equation}
Indeed, (\ref{eq.zv}) implies that
\begin{equation*}\label{eq.prN1}
x_{2n}\(\pi\(1-\frac 1{n}\)\)=\pi\(1-\frac 1{n}\)+\sum_{\nu=1}^{2n}
\binom {2n}{\nu} \frac{(-1)^\nu}{\nu} \sin \pi \(1-\frac 1n\)\nu.
\end{equation*}
Using the formula
$
\binom{2n}{\nu}=\binom{2n}{n-\nu}=\binom{2n}{\nu+n},
$
we derive
\begin{equation*}
\begin{split}
\sum_{\nu=1}^{2n} \binom {2n}\nu \frac{(-1)^\nu}{\nu} {\sin \pi \(1-\frac
1n\)}{\nu}&=-\sum_{\nu=1}^{2n} \binom {2n} \nu\frac 1\nu \sin \frac{\pi
\nu}{n}\\
&=-\sum_{\nu=1}^{n}\binom {2n}\nu \(\frac 1\nu-\frac
1{\nu+n}\)\sin\frac{\pi \nu}{n}.
\end{split}
\end{equation*}
Combining these equalities and the well-known asymptotic of the binomial coefficients
$$
\binom{2n}{n}\sim \frac{4^n}{\sqrt {\pi n}}\quad\text{as}\quad n\to\infty,
$$
we obtain that (\ref{eqkvad}) is fulfilled.

Now, the decreasing of the function $x_\b(t)$ on $(\pi\(1-2/\b\),\pi\(1-1/\b\))$  and $(\ref{eqkvad})$ imply that for any $N_0\in
2\N$ one can find $N_1\in 2\N$ such that
$$
x_{N_1}(\t({N_1}))<x_{N_1}\(\pi\(1-\frac
2{N_1}\)\)<x_{N_0}(\t({N_0})).
$$
At that, choosing $N_1$ to be sufficiently large number, we obtain that
$$
x_{N_0}(\t({N_0}))-x_{N_1}(\t({N_1}))>2\pi.
$$
Then, taking into account~\eqref{eq4.lem1}, one can choose
 $k_1\in \N$ such that
\begin{equation}\label{oo1}
  x_{N_1}(\t({N_1})+2\pi k_1)<0\quad\text{and}\quad
x_{N_0}(\t({N_0})+2\pi k_1)>0.
\end{equation}

Denote
$$
F_1(\b)=x_\b(\t(\b)+2\pi k_1).
$$
By analogy with the above proof of $(ii)$ in the case $k=0$, we can show that the function $F_1(\b)=x_\b(\t(\b)+2\pi k_1)$ is continuous.
Then, taking into account~\eqref{oo1} and applying the intermediate value theorem to $F_1$,  we can find the points  $t_1>\pi (1-  2/{N_0})+2\pi k_1$ and
$\b_1>\b_0$ such that  $z_{\b_{1}}(t_1)=0$.

Repeating this scheme, we obtain infinite sets of points  $\{\b_k\}_{k\in\Z_+}$ and $\{t_k\}_{k\in\Z_+}$, for which assertion  $(ii)$ holds.

\medskip

\emph{{The proof of $(iii)$.}}
Suppose to the contrary that there exists $t_0\in (0,\pi)$ such that for some $\b>4$ one has
\begin{equation}\label{eqlem2iii.3}
    \psi_\b(t_0)=0.
\end{equation}
Then from  (\ref{eq5.lem1}) and (\ref{eq6.lem1}), we get
\begin{equation*}\label{eqlem2iii.4}
    \int_0^{\frac \b 2 t_0}e^{\i \vp}\(\sin \frac\vp\b\)^\b {\rm d}\vp=0.
\end{equation*}
This equality  implies that
\begin{equation*}
    \int_0^{\frac \b 2 t_0}\sin(\vp-\vp_0) \(\sin \frac\vp\b\)^\b {\rm d}\vp =0
\end{equation*}
and, therefore, one has
\begin{equation}\label{eqlem2iii.5}
    \int_{\vp_0}^{2\pi l_0}\sin \vp\(\sin \frac{\vp-\vp_0}\b\)^\b
    {\rm d}\vp=0,
\end{equation}
where we put
$
\vp_0=2\pi l_0-\b t_0/2
$
and $l_0=\min\{l\in \N\,:\, \b
t_0/2\le 2\pi l\}$.

From~\eqref{eqlem2iii.5}, we obtain
\begin{equation}\label{eqlem2iii.6}
\begin{split}
     \int_{\vp_0}^{2\pi}&\(\sin \frac{\vp-\vp_0}\b\)^\b \sin \vp
    {\rm d}\vp\\
    &+\sum_{j=1}^{l_0-1}\int_{2\pi j}^{2\pi (j+1)}\(\sin \frac{\vp-\vp_0}\b\)^\b \sin \vp
    {\rm d}\vp=0
\end{split}
\end{equation}
(if $l_0=1$, then this sum contains only one term).
It is easy to see that all summands in the above sum are negative.
Indeed, it is obvious that for $\vp_0 \in [\pi,2\pi)$ we have
$$
\int_{\vp_0}^{2\pi}\(\sin \frac{\vp-\vp_0}\b\)^\b \sin \vp
    {\rm d}\vp<0.
$$
At the same time, if $\vp_0 \in [0,\pi)$, then
\begin{equation*}
    \begin{split}
\int_{\vp_0}^{2\pi}\(\sin \frac{\vp-\vp_0}\b\)^\b \sin \vp
    {\rm d}\vp=\left\{\int_{\vp_0}^{\pi}+\int_{\pi}^{2\pi}\right\}\(\sin \frac{\vp-\vp_0}\b\)^\b \sin \vp
    {\rm d}\vp\\
<\(\sin \frac{\pi-\vp_0}\b\)^\b\int_{\vp_0}^{\pi}\sin\vp {\rm d}\vp+\(\sin
\frac{\pi-\vp_0}\b\)^\b \int_{\pi}^{2\pi}\sin\vp {\rm d}\vp<0.
     \end{split}
\end{equation*}
By analogy, we can prove that for any $l_0>1$ and
$j=1,\dots,l_0-1$ one has
$$
\int_{2\pi j}^{2\pi (j+1)}\(\sin \frac{\vp-\vp_0}\b\)^\b \sin \vp
    {\rm d}\vp<0.
$$
Thus, the sum in  (\ref{eqlem2iii.6}) is negative, which is a contradiction to  (\ref{eqlem2iii.3}).
\end{proof}


\subsection{Properties of Fourier multipliers}
To prove Theorems~\ref{th4} and~\ref{th5}, we need some facts about Fourier multipliers $\{\l_k\}_{k\in \Z}$.

In the case $\l_k=g(\e k)$, $g\in C(\R)$, $\e>0$, it is important to ascertain whether the function $g$ lies in the Banach algebra
$$
B(\R)=\bigg\{g\,:\,g(t)=\int_{-\infty}^\infty e^{-\i t u}{\rm d}\mu(u),\quad
\Vert g\Vert_B={\rm var} \mu<\infty\bigg\},
$$
where $\mu$ is a complex Borel measure which is finite on  $\R$, and
$\var\mu$ is the total
variation of~$\mu$. The point is that for any $1\le p\le\infty$, we have
$$
\sup_{\e>0}\Vert \{g(\e k)\}\Vert_{M_p}\le\sup_{\e>0}\Vert \{g(\e k)\}\Vert_{M_1}=\Vert g\Vert_B
$$
(see \cite{SW} or \cite[Ch. 7]{TB}).

In what follows, we will use the following comparison principle (see \cite[7.1.11 and 7.1.14]{TB}).
\begin{lemma}\label{lem3}
Let $\vp$ and $\psi$ belong to $C(\R)$ and
$$
\Phi_\e (f;x)\sim\sum_{k\in\Z} \vp(\e k) \widehat{f}_k e^{\i kx},\quad
\Psi_\e(f;x)\sim\sum_{k\in\Z} \psi(\e k) \widehat{f}_k e^{\i kx},\quad f\in L_1(\T),
$$
where the series on the right are Fourier series.
If $g=\vp/\psi\in B(\R)$, then for all $f\in L_p(\T)$, $1\le p\le\infty$, and $\e>0$ we have
$$
\Vert \Phi_\e(f)\Vert_p \le \Vert g\Vert_B \Vert \Psi_\e(f)\Vert_p.
$$
\end{lemma}

To verify that the function $g$ belongs to $B(\R)$, one can use the following Beurling's sufficient condition (see~\cite[6.4.2]{TB} or the survey~\cite{LST}).

\begin{lemma}\label{lem4}
  Let a function $g$ be locally absolutely continuous on $\R$, $g\in L_2(\R)$, and $g'\in
    L_2(\T)$. Then
    $$
    \Vert g\Vert_B\le C\(\Vert g\Vert_{L_2(\R)}+\Vert
    g'\Vert_{L_2(\R)}\).
    $$
\end{lemma}

\section{Proofs of the main results}

\setcounter{section}{4}
\setcounter{equation}{0}\setcounter{theorem}{0}

\begin{proof}[Proof of Theorem~\ref{th1}]

The inequality ${\rm w}_\b(f,h)_p\le C{\w}_\b(f,h)_p$   is obvious.  Let us prove the converse inequality.

Let $n\in \N$ be such that $n=[1/h]$ and let $T_n\in\mathcal{T}_n$ be polynomials of the best approximation of $f$ in $L_p(\T)$, that is $\Vert f-T_n\Vert_p=E_n(f)_p$.

By properties ${(e)}$ and ${(f)}$ of the modulus of smoothness and Lemma~\ref{lemNS} for $\d\in (h/2,h)$, we obtain
\begin{equation}\label{pro1}
  \begin{split}
     \w_\b(f,h)_p^{p_1}&\le \w_\b(f-T_n,h)_p^{p_1}+\w_\b(T_n,h)_p^{p_1}\\
            &\le C\(\Vert f-T_n\Vert_p^{p_1}+\(\frac h\d\)^{\b p_1}\Vert \D_\d^\b T_n\Vert_p^{p_1}\)\\
            &\le C\(\Vert f-T_n\Vert_p^{p_1}+\Vert \D_\d^\b f\Vert_p^{p_1}\).
  \end{split}
\end{equation}
Integrating inequality~\eqref{pro1} by $\d$ over $(h/2,h)$ and applying the Jackson inequality~\eqref{eqI12} and ${(g)}$, we get
\begin{equation}\label{pro2}
  \begin{split}
     \w_\b(f,h)_p^{p_1}&\le  C\(E_n(f)_p^{p_1}+\frac1h\int_{h/2}^h\Vert \D_\d^\b f\Vert_p^{p_1}{\rm d}\d\)\\
     &\le  C\(\w_r\(f,\frac1n\)_p^{p_1}+\frac1h\int_{0}^h\Vert \D_\d^\b f\Vert_p^{p_1}{\rm d}\d\)\\
     &\le  C\(\w_r\(f,h\)_p^{p_1}+\frac1h\int_{0}^h\Vert \D_\d^\b f\Vert_p^{p_1}{\rm d}\d\),
  \end{split}
\end{equation}
where we take $r=\b+\a\in \N$ with $\a>1/p_1-1$.

Next, using the first equivalence from~\eqref{equivmmm} for all $0<p\le\infty$ and ${(c)}$, we obtain
\begin{equation}\label{pro3}
  \begin{split}
    \w_r\(f,h\)_p^{p_1}&\le \frac Ch\int_{0}^h\Vert \D_\d^r f\Vert_p^{p_1}{\rm d}\d=\frac Ch\int_{0}^h\Vert \D_\d^\a\D_\d^\b f\Vert_p^{p_1}{\rm d}\d\\
    &\le \frac Ch\int_{0}^h\Vert \D_\d^\b f\Vert_p^{p_1}{\rm d}\d.
  \end{split}
\end{equation}

Finally, combining~\eqref{pro2} and~\eqref{pro3}, we get ${\w}_\b(f,h)_p\le C{\rm w}_\b(f,h)_p$.
\end{proof}

\begin{proof}[Proof of Theorem~\ref{th3}]
\emph{{The proof of $(i)$}} easily follows from Lemma~\ref{lem0}, Lem\-ma~\ref{lem1} $(i)$, and \eqref{equivmmm}.

\emph{{The proof of $(ii)$}.}
By~\eqref{mainf0}, we have
$$
\frac 1h\int_0^{h} \D_\d^{\b_k} e_n(x){\rm d}\d=\psi_{\b_k}({\rm sign\,}n\cdot t_k)e_n(x).
$$
It remains only to take into account that by Lemma~\ref{lem1} $(ii)$ and equalities~\eqref{+++} one has $\psi_{\b_k}(\pm t_k)=0$ .
\end{proof}

\begin{proof}[Proof of Theorem~\ref{th4}]
It is obvious that
$     \widetilde{\w}_\b(T_n,h)_p \le \w_\b(T_n,h)_p$.
Let us show
\begin{equation}\label{eqprth3.1}
    \w_\b(T_n,h)_p\le C \widetilde{\w}_\b(T_n,h)_p\quad\text{for all} \quad h\in (0,1/n).
\end{equation}

Denote
$$
g_\t(t)=\frac{(1-e^{\i \t t})^\b v(t)}{\psi_\b(t)},\quad \t\in (0,1),
$$
where $v\in C^\infty(\R)$, $v(t)=1$ for $|t|\le 1$ and $v(t)=0$ for
$|t|> 2$. Note that
\begin{equation}\label{form}
  \D_\d^\b f(x)\sim\sum_{k\in\Z}(1-e^{\i k\d})^\b \widehat{f}_k e^{\i kx},\quad f\in L_1(\T).
\end{equation}
Thus, by Lemma~\ref{lem3}, to prove (\ref{eqprth3.1}) it is enough to verify
\begin{equation}\label{eqprth3.2}
 \sup_{\t \in (0,1)}\Vert g_\t\Vert_B<\infty.
\end{equation}

By Lemma~\ref{lem1} $(iii)$ we have that
$\psi_\b(t)\neq 0$ for $0<|t|<\pi$.
Moreover, it is easy to see that $g_\t
\in C^\infty(\R)$ and
$\sup_{\t \in (0,1)}\Vert g_\t^{(k)}\Vert_{L_\infty(\R)}<\infty$ for each $k\in \Z_+$.
Thus, using Lemma~\ref{lem4} we see that
(\ref{eqprth3.2}) holds.
\end{proof}

\begin{proof}[Proof of Theorem~\ref{thbest}]
In view of~\eqref{lemmm}, we only need to verify that
\begin{equation}\label{preqthEqmods1}
{\w}_\b(f,h)_p\le C\(\widetilde{\w}_\b(f,h)_p+E_{[1/h]}(f)_p\).
\end{equation}
Let $n=[1/h]$ and let $T_n\in\mathcal{T}_n$ be such that $\Vert f-T_n\Vert_p=E_n(f)_p$. Using properties $(e)$ and $(f)$ and Theorem~\ref{th4}, we obtain
\begin{equation*}
  \begin{split}
     \w_\b(f,h)_p&\le \w_\b(f-T_n,h)_p+\w_\b(T_n,h)_p\\
            &\le C\Vert f-T_n\Vert_p+\widetilde{\w}_\b(T_n,h)_p\\
            &\le C\Vert f-T_n\Vert_p+\widetilde{\w}_\b(f,h)_p\\
            &\le C\(\widetilde{\w}_\b(f,h)_p+E_{[1/h]}(f)_p\).
  \end{split}
\end{equation*}
Thus, we have~\eqref{preqthEqmods1}.
\end{proof}

\begin{proof}[Proof of Theorem~\ref{th5}]

First let us prove that
\begin{equation}\label{ab1}
  \w_{\b;\, \a}^*(f,h)_p\le C\w_\b(f,h)_p.
\end{equation}

Let $n=[1/h]$ and let $T_n\in\mathcal{T}_n$ be such that $\Vert f-T_n\Vert_p=E_n(f)_p$. Then, by properties  $(a), (b), (c)$, and $(g)$,  and the Jackson inequality~\eqref{eqI12}, we derive
\begin{equation}\label{ab2}
\begin{split}
  \w_{\b;\, \a}^*(f,h)_p &\le \w_{\b;\, \a}^*(f-T_n,h)_p+\w_{\b;\, \a}^*(T_n,h)_p\\
  &\le C\Vert f-T_n\Vert_p+\w_{\b;\, \a}^*(T_n,h)_p\\
  &\le C\w_\b(f,h)_p+\w_{\b;\, \a}^*(T_n,h)_p.
\end{split}
\end{equation}
Thus, it remains to show
\begin{equation}\label{ab3}
  \w_{\b;\, \a}^*(T_n,h)_p\le C\w_{\b}(f,h)_p.
\end{equation}

By Lemma~\ref{lemNS} with $0<\d_1, \d_2<h$ and by $(b)$ and $(c)$, we obtain
\begin{equation}\label{ab4}
  \begin{split}
     \Vert \D_{\d_1}^{\b-\a}\D_{\d_2}^{\a}T_n\Vert_p&\le C \Vert \D_{h}^{\b-\a}\D_{\d_1}^{\a}T_n\Vert_p= C \Vert \D_{\d_1}^{\a}\D_{h}^{\b-\a}T_n\Vert_p\\
     &\le C\Vert \D_{h}^{\b}T_n\Vert_p\le C\w_{\b}(T_n,h)_p.
   \end{split}
\end{equation}
At the same time, as above, by $(e)$, $(f)$, and~\eqref{eqI12}, we get
\begin{equation}\label{ab5}
  \w_{\b}(T_n,h)_p\le C\Vert f-T_n\Vert_p+\w_{\b}(f,h)_p\le C\w_{\b}(f,h)_p.
\end{equation}
Now, by the definition of $\w_{\b;\, \a}^*(f,h)_p$, \eqref{ab4}, and \eqref{ab5}, inequality~\eqref{ab3} easily follows.

Combining ~\eqref{ab2} and~\eqref{ab3}, we get~\eqref{ab1}.

To prove the converse inequality
\begin{equation}\label{abc1}
  \w_\b(f,h)_p\le C\w_{\b;\, \a}^*(f,h)_p,
\end{equation}
we  use the de la Vall\'ee-Poussin means of $f$ given by
$$
V_h(f;x)=\sum_{k\in \Z} v\(kh\)\widehat{f}_k e^{\i kx},
$$
where the function $v$ is defined in the proof of Theorem~\ref{thbest} (in addition we suppose that $0\le v(t)\le 1$ for all $t\in\R$).

By property $(e)$, we have
\begin{equation*}
  \w_\b(f,h)_p\le \w_\b(V_h(f),h)_p+\w_\b(f-V_h(f),h)_p.
\end{equation*}
Therefore, the proof of~\eqref{abc1} will follows from the following two inequalities
\begin{equation*}
  \w_\b(V_h(f),h)_p\le C \w_{\b;\, \a}^*(f,h)_p
\end{equation*}
and
\begin{equation}\label{abc4}
  \w_\b(f-V_h(f),h)_p\le C \w_{\b;\, \a}^*(f,h)_p.
\end{equation}

The first inequality can be verified by repeating the proof of Theorem~\ref{th4} with the function
$$
g_{1,\t}(t)=\frac{ (1-e^{\i\t t})^\b v(t)}{\psi_{\b-\a}(t)\psi_\a(t)}
$$
instead of $g_\t$.

Let us verify that \eqref{abc4} holds. By Lemma~\ref{lem3} and \eqref{form} it is enough to check that
\begin{equation*}
  g_{2,\t}(t)=\frac{(1-e^{\i\t t})^\b(1-v(t))}{\psi_{\b-\a}(t)\psi_\a(t)}\in B(\R)
\end{equation*}
and
\begin{equation}\label{abc6}
 \sup_{\t \in (0,1)}\Vert g_{2,\t}\Vert_B<\infty.
\end{equation}
Indeed, it is easy to see that
\begin{equation}\label{abc7}
 \Vert (1-e^{\i\t (\cdot)})^\b\Vert_B\le 2^\b.
\end{equation}
At the same time we have that $\psi_\a\in B(\R)$, $\psi_\a(t)\neq 0$ for $t\in \R\setminus\{0\}$, and $\lim_{|t|\to\infty}\psi_\a(t)=1$. Thus, by the Wiener-L\'evy theorem (see~\cite[Theorem~4.4]{LST}), we get
\begin{equation}\label{abc8}
\frac{(1-v(t))^{1/2}}{\psi_\a(t)}\in B(\R).
\end{equation}
By analogy, we have
\begin{equation}\label{abc9}
\frac{(1-v(t))^{1/2}}{\psi_{\b-\a}(t)}\in B(\R).
\end{equation}
Finally, combining \eqref{abc7}--\eqref{abc9} and taking into account that $B(\R)$ is the Banach algebra, we derive \eqref{abc6} and, therefore,~\eqref{abc4}.

Theorem~\ref{th5} is proved.
\end{proof}

\section*{Acknowledgements}

The author thanks to R.M. Trigub for the proposed problem and to A.A.~Dovgoshey for the valuable discussions.

This research has received funding from the European Union's Horizon 2020 research and innovation
programme under the Marie Sklodowska-Curie grant agreement No 704030.



\end{document}